\documentclass[a4paper, 10pt]{amsart}
\usepackage{amsmath, amssymb, hyperref}

\theoremstyle{plain}
\newtheorem{theorem}{\bf Theorem}[section]
\newtheorem{proposition}[theorem]{\bf Proposition}

\theoremstyle{definition}
\newtheorem{example}[theorem]{\bf Example}

\newtheorem{remark}[theorem]{\bf Remark}

\newcommand{\N}{\mathbb N}
\newcommand{\Z}{\mathbb Z}

\newcommand{\Q}{\mathbb Q}

 \DeclareMathOperator{\ord}{ord}

\newcommand{\red}{{\text{\rm red}}}
\newcommand{\C}{\text{\rm C}}

\newcommand{\bdot}{\boldsymbol{\cdot}}
\newcommand{\DP}{\negthinspace : \negthinspace}

\numberwithin{equation}{section}

\begin{document}

\title{A characterization of seminormal C-monoids}

\author{Alfred Geroldinger and Qinghai Zhong}

\address{University of Graz, NAWI Graz \\
Institute for Mathematics and Scientific Computing \\
Heinrichstra{\ss}e 36\\
8010 Graz, Austria}

\email{alfred.geroldinger@uni-graz.at, qinghai.zhong@uni-graz.at}
\urladdr{http://imsc.uni-graz.at/geroldinger, http://qinghai-zhong.weebly.com/}



\thanks{This work was supported by the Austrian Science Fund FWF, Project Number P28864-N35}

\begin{abstract}
It is well-known that a C-monoid is completely integrally closed if and only if its reduced class semigroup is a group and if this holds, then the C-monoid is a Krull monoid and the reduced class semigroup coincides with the usual class group of Krull monoids. We prove that a C-monoid is seminormal if and only if its reduced class semigroup  is a union of groups. Based on this characterization we establish a criterion (in terms of the class semigroup) when seminormal C-monoids are half-factorial.
\end{abstract}

\maketitle

\section{Introduction} \label{1}
\smallskip

A C-monoid $H$ is a submonoid of a factorial monoid, say $H \subset F$,  such that $H^{\times} = H \cap F^{\times}$ and the reduced class semigroup is finite. A commutative ring is a C-ring if its multiplicative monoid of regular elements is a C-monoid. Every C-monoid is Mori (i.e., $v$-noetherian), its complete integral closure $\widehat H$ is a Krull monoid with finite class group $\mathcal C (\widehat H)$, and the conductor $(H \DP \widehat H)$ is non-trivial. Conversely, every  Mori domain $R$ with non-zero conductor $\mathfrak f = (R \DP \widehat R)$, for which  the residue class ring $R/\mathfrak f$ and the class group $\mathcal C (\widehat R)$ are finite,  is a C-domain (\cite[Theorem 2.11.9]{Ge-HK06a}), and these two finiteness conditions are equivalent to being a C-domain for   non-local semilocal noetherian domains (\cite[Corollary 4.5]{Re13a}).
The following result is well-known (see Section \ref{2} where we gather the basics on C-monoids).

\medskip
\noindent
{\bf Theorem A.} Let $H$ be a  C-monoid. Then $H$ is completely integrally closed if and only if its reduced class semigroup is a group. If this holds, then $H$ is a Krull monoid and the reduced class semigroup coincides with the class group of $H$.

\smallskip
The goal of the present note is the following characterization of seminormal C-monoids. Note that all seminormal C-domains are seminormal Mori domains (see \cite{Ba94} for a survey) and that, in particular, all seminormal orders in algebraic number fields (see \cite{Do-Fo87}) are seminormal C-domains.

\medskip
\begin{theorem} \label{1.1}
Let $H$ be a \C-monoid. Then $H$ is seminormal if and only if its reduced class semigroup is a union of groups. If this holds, then we have
\begin{enumerate}
\item The class of every element of $H$ is an idempotent of the class semigroup.

\item The constituent group of the smallest idempotent of the class semigroup is isomorphic to the class group of the complete integral closure of $H$.
\end{enumerate}
\end{theorem}

The finiteness of the class semigroup allows to establish a variety of arithmetical finiteness results for C-monoids (for a sample, out of many, see \cite[Theorem 4.6.6]{Ge-HK06a} or \cite{Fo-Ha06a}). In the special case of Krull monoids with finite class group, precise arithmetical results can be given in terms of the structure of the class group (for a survey see \cite{Sc16a}). As a first step in this direction in the more general  case of seminormal C-monoids, we establish -- based on Theorem \ref{1.1} -- a characterization of half-factoriality in terms of the class  semigroup (Theorem \ref{4.2}).

\medskip
\section{Background on C-monoids} \label{2}
\medskip

We denote by $\N$ the set of positive integers and set $\N_0 = \N \cup \{0\}$. For rationals $a, b \in \Q$, we denote by $[a,b] = \{x \in \Z \mid a \le x \le b\}$ the discrete interval lying between $a$ and $b$.
All semigroups and rings in this paper are commutative and have an identity element and all homomorphisms respect identity elements.
Let $S$ be a multiplicatively written semigroup. Then $S^{\times}$ denotes its group of invertible elements and  $\mathsf E (S)$ its set of  idempotents. For subsets $A, B \subset S$ and an element $c \in S$ we set
\[
AB = \{ab \mid a \in A, b \in B \} \quad \text{and} \quad cB = \{cb \mid b \in B \} \,.
\]

Let $\mathcal C$ be  semigroup. We use additive notation and have class groups and class semigroups in mind. Let $e, f \in \mathsf E (\mathcal C)$. We denote by $\mathcal C_e$  the set of all $x \in \mathcal C$ such that $x+e=x$ and $x+y=e$ for some $y \in \mathcal C$. Then $\mathcal C_e$ is a group with identity element $e$, called the {\it constituent group} of $e$. If $e \ne f$, then $\mathcal C_e \cap \mathcal C_f = \emptyset$. An element $x \in \mathcal C$ is contained in a subgroup of $\mathcal C$ if and only if there is some $e \in \mathsf E (\mathcal C)$ such that $x \in \mathcal C_e$. Thus $\mathcal C$ is a union of groups if and only if
\begin{equation} \label{structure-Clifford}
\mathcal C = \bigcup_{e \in \mathsf E (\mathcal C)} \mathcal C_e \,.
\end{equation}
Semigroups with this property are called {\it Clifford semigroups} (\cite{Gr01, HK09a}).
Clearly, $0 \in \mathsf E (\mathcal C)$ and $\mathcal C_0 = \mathcal C^{\times}$. The Rees order $\le$ on $\mathsf E (S)$ is defined by $e \le f$ if $e+f = e$. If $\mathsf E (S) = \{0=e_0, \ldots, e_n \}$, then $e_0$ is the largest element and $e_0 + \ldots + e_n$ is the smallest idempotent in the Rees order.
If $\mathcal C$ is finite, then for every $a \in \mathcal C$ there is an $n \in \N$ such that $na = a + \ldots + a \in \mathsf E (\mathcal C)$.

By a {\it monoid}, we mean a cancellative semigroup. We use multiplicative notation and have monoids of nonzero elements of domains in mind.
Let $H$ be a monoid, $H_{\red} = \{aH^{\times} \mid a \in H\}$ the associated reduced monoid,  and $\mathsf q (H)$ the quotient group of $H$.  We denote by
\begin{itemize}
\item $H' = \{ x \in \mathsf q (H) \mid \ \text{there is an $m \in \N$ such that } \ x^n \in H \ \text{for all } \ n \ge m \}$ the {\it seminormalization} of $H$ and by

\item $\widehat H = \{ x \in \mathsf q (H) \mid \ \text{there is a $c \in H$ such that } \ cx^n \in H \ \text{for all} \ n \in \N \}$ the {\it complete integral closure} of $H$.
\end{itemize}
Then $H \subset H' \subset \widehat H \subset \mathsf q (H)$ and $H$ is said to be
\begin{itemize}
\item {\it seminormal} if $H = H'$,
\item {\it completely integrally closed} if $H = \widehat H$,
\item $v$-{\it noetherian} (resp. {\it Mori}) if H satiesfies the ACC on divisorial ideals,
\item a {\it Krull monoid} if H is a completely integrally closed  Mori monoid.
\end{itemize}

Let $F$ be a factorial monoid, say $F = F^{\times} \times \mathcal F (P)$, where $\mathcal F (P)$ is the free abelian monoid with basis $P$. Then every $a \in F$ can be written in the form
\[
a = \varepsilon \prod_{p \in P} p^{\mathsf v_p (a)} \,,
\]
where $\varepsilon \in F^{\times}$ and $\mathsf v_p  \colon F \to \N_0$ is the $p$-adic exponent. Let $H \subset F$ be a submonoid. Two elements $y, y' \in F$ are called $H$-equivalent (we write $y \sim y'$) if $y^{-1}H \cap F = {y'}^{-1} H \cap F$ or, in other words,
\[
\text{if for all} \ x \in F, \ \text{we have} \ xy \in H \quad \text{if and only if} \quad xy' \in H \,.
\]
$H$-equivalence defines a congruence relation on $F$ and for $y \in F$, let $[y]_H^F = [y]$ denote the congruence class of $y$. Then
\[
\mathcal C (H,F) = \{[y] \mid y \in F \} \quad \text{and} \quad \mathcal C ^* (H,F) = \{ [y] \mid y \in (F \setminus F^{\times}) \cup \{1\} \}
\]
are commutative semigroups with identity element $[1]$ (introduced in \cite[Section 4]{Ge-HK04a}). $\mathcal C (H,F)$ is the {\it class semigroup} of $H$ in  $F$ and the subsemigroup $\mathcal C^* (H,F) \subset \mathcal C (H,F)$ is the {\it reduced class semigroup} of $H$ in $F$. We have
\begin{equation} \label{reduced}
\mathcal C (H,F) = \mathcal C^* (H,F) \cup \{[y] \mid y \in F^{\times}\} \ \text{and} \ \{[y] \mid y \in F^{\times}\} \cong F^{\times}/H^{\times} \,.
\end{equation}
It is easy to check that either $\{[y] \mid y \in F^{\times} \} \subset \mathcal C^* (H,F)$ or $\mathcal C^* (H,F) \cap \{[y] \mid y \in F^{\times} \} = \{[1]\}$.
Thus
\begin{equation} \label{group}
\text{$\mathcal C (H,F)$ is a union of groups if and only if $\mathcal C^* (H,F)$ is a union of groups}
\end{equation}
 and $\mathcal C (H,F)$ is finite if and only if both $\mathcal C^* (H,F)$ and $F^{\times}/H^{\times}$ are finite. As usual, (reduced) class semigroups and class groups will be written additively.

A monoid homomorphism $\partial \colon H \to \mathcal F (P)$ is called a {\it divisor theory} (for $H$) if the following two conditions hold:
\begin{itemize}
\item If $a, b \in H$ and $\varphi (a) \mid \varphi (b)$ in $\mathcal F (P)$, then $a \mid b$ in $H$.
\item For every $\alpha \in \mathcal F (P)$ there are $a_1, \ldots, a_m \in H$ such that $\alpha = \gcd \big( \varphi (a_1), \ldots, \varphi (a_m) \big)$.
\end{itemize}
A monoid has a divisor theory if and only if it is a Krull monoid (\cite[Theorem 2.4.8]{Ge-HK06a}). Suppose $H$ is a Krull monoid. Then there is an embedding of $H_{\red}$ into a free abelian monoid, say $H_{\red} \hookrightarrow \mathcal F (P)$. In this case $\mathcal F (P)$ is uniquely determined (up to isomorphism) and
\[
\mathcal C (H) = \mathsf q ( \mathcal F (P))/ \mathsf q (H_{\red})
\]
is the (divisor) class group of $H$ which  is isomorphic to the $v$-class group $\mathcal C_v (H)$ of $H$.

A monoid $H$ is a C-{\it monoid} if it is a submonoid of some factorial monoid $F = F^{\times} \times \mathcal F (P)$ such that $H^{\times} = H \cap F^{\times}$ and the reduced class semigroup is finite.
We say that $H$ is {\it dense} in $F$ if $\mathsf v_p (H) \subset \N_0$ is a numerical monoid for all primes $p$ of $F$.
Suppose that $H$ is a C-monoid. Proofs of the following facts can be found in (\cite[Theorems 2.9.11 and 2.9.12]{Ge-HK06a}).  The monoid $H$ is $v$-noetherian, the conductor $(H \DP \widehat H) \ne \emptyset$, and there is a factorial monoid $F$ such that $H \subset F$ is dense. Suppose that $H \subset F$ is  dense.  Then the map
\begin{equation} \label{divisor-theory}
\partial \colon \widehat H \to \mathcal F (P), \quad \text{defined by} \quad  \partial (a) = \prod_{p \in P} p^{\mathsf v_p (a)}
\end{equation}
is a divisor theory. In particular, $\widehat H$ is a Krull monoid,   $F_{\red}$ and hence $\mathcal C^* (H,F)$ are uniquely determined by $H$, and we call $\mathcal C^* (H,F)$ {\it the reduced class semigroup of $H$}.
\begin{equation} \label{Krull}
\text{If $\mathcal C^* (H,F)$ is a group, then $H$ is a Krull monoid}
\end{equation}
and every Krull monoid with finite class group is a C-monoid.

Let $R$ be a (commutative integral) domain. Then $R$ is seminormal (completely integrally closed, Mori) if and only if its multiplicative monoid $R^{\bullet}$ of nonzero elements has the respective property. Moreover, $R$
\begin{itemize}
\item is a {\it Krull domain} if $R^{\bullet}$ is a Krull monoid, and
\item is a {\it  \C-domain} if $R^{\bullet}$ is a C-monoid.
\end{itemize}
Both statements generalize to rings with zero-divisors (\cite[Theorem 3.5 and Section 4]{Ge-Ra-Re15c}).  We refer to \cite{HK-Ha-Ka04, Ba-Ch14a, Cz-Do-Ge16a, Oh18a} for C-monoids, that do not stem from ring theory, and to \cite{Ka16b} for a more general concept.

\medskip
\section{A characterization of seminormality} \label{3}
\medskip

In this section we first prove the characterization of seminormality, as formulated in Theorem \ref{1.1}. After that we discuss a special class of C-monoids, namely the monoid of product-one sequences over a finite group.

\begin{proof}[Proof of Theorem \ref{1.1}]
By \cite[Theorem 2.9.11.4]{Ge-HK06a}, we may choose a special factorial monoid $F$ such that $H \subset F$ is a dense C-monoid. This special choice will turn out to be useful in the proof of part 3. Since $\widehat H$ is a Krull monoid, we have $\widehat H = \widehat H^{\times} \times H_0$, where $H_0$ is a reduced Krull monoid. If the embedding $H_0 \hookrightarrow \mathcal F (P)$ is a divisor theory, then $H \subset F = \widehat H^{\times} \times \mathcal F (P)$ is a dense C-monoid. In particular, we have
$H^{\times} = H \cap F^{\times}$ and $\mathcal C^* = \mathcal C^* (H,F)$ is finite. Furthermore,  there are $\alpha \in \N$ and a subgroup $V\subset F^{\times}$ such that the following properties hold (\cite[2.9.5 and 2.9.6]{Ge-HK06a}:

\begin{itemize}
\item[P1.] $H^\times \subset V\,, \quad (F^\times \DP V) \mid \alpha\,, \quad V(H \setminus H^\times) \subset H$,

\item[P2.]  $q^{2\alpha}F \cap H = q^\alpha (q^\alpha F \cap H) \quad \text{for all} \quad q \in F \setminus F^\times$ \,.
\end{itemize}
In particular, if $p \in P$ and $a \in p^\alpha F$, then $a \in H$ if and only if $p^\alpha a \in H$.

\smallskip
1. First we suppose that $H$ is seminormal. We have to  show that  $\mathcal C^*$ is a union of groups.
We choose an element $a \in F\setminus F^{\times}$, say $a = \epsilon p_1^{k_1}\cdot\ldots\cdot p_t^{k_t}$, where $\epsilon\in F^{\times}$, $p_1,\ldots,p_t\in P$, $t, k_1,\ldots,k_t \in \N$.
Since  $\mathcal C^*$ is finite, there is an $n \in \N$ such that $[a^n]$ is an idempotent element of $\mathcal C^*$. Since for every multiple $n'$ of $n$, $[a^{n'}]$ is  an idempotent element of $\mathcal C^*$, we may assume without restriction that $n$ is a multiple of $\alpha$.
We assert that $a^{n+1} \sim a$. Clearly, this implies that $[a]$ lies in a cyclic subgroup of $\mathcal C^*$.

In order to show that $a^{n+1} \sim a$, we choose an element $b \in F$, . First, suppose that $ab \in H$. Then for every $r\ge \alpha$, we have $(ab)^r\in H\cap p_1^{\alpha}\cdot\ldots\cdot p_t^{\alpha} F$. It follows by Properties (P1) and (P2) that
\[
(a^{n+1}b)^r=a^{rn}(ab)^r=\epsilon^{rn}p_1^{rk_1\alpha}\cdot\ldots\cdot p_t^{rk_t\alpha}(ab)^r\in V(H\setminus H^{\times})\subset H\,.
\]
Since $H$ is seminormal, we obtain $a^{n+1}b\in H$.

Conversely, suppose that $ba^{n+1} \in H$. We have to verify that $ba \in H$. To do so we claim that
\begin{equation} \label{technical}
(ba)^m a^n \in H \quad \text{for all } \ m \in \N \,,
\end{equation}
and we proceed by induction on $m$. If $m=1$, then this holds by assumption. Suppose the claim holds for $m \in \N$. Then, by the induction hypothesis,
\[
(ba)^{m+1} a^n \sim (ba)^{m+1}a^{2n}= \Big( (ba)^m a^n \Big) \Big( b a^{n+1} \Big) \in H
\]
whence $(ba)^{m+1}a^n \in H$. Using \eqref{technical} with $m=n$ we infer that
\[
(ba)^n =b^na^n\sim b^na^{2n}=(ba)^n a^n \in H
\]
whence $(ba)^n \in H$. Thus we obtain that
\[
(ba)^{n+1} =(b^{n+1}a)a^n\sim (b^{n+1}a)a^{2n}= (ba)^n \Big(b a^{n+1} \Big) \in H
\]
whence $(ba)^{n+1} \in H$. This implies that $(ba)^{\ell} \in H$ for all sufficiently large $\ell \in \N$. Since $H$ is seminormal, it follows that $ba \in H$.

\smallskip
2. Now we suppose that   $\mathcal C^*$ is a union of groups. Then, by \eqref{reduced}, the class semigroup  $\mathcal C (H,F)$ is a union of groups.
In order to show that $H$ is seminormal, let $a \in \mathsf q (H)$ be given such that $a^n\in H$ for all $n\ge N$ for some $N \in \N$. Since $F$ is factorial, it is seminormal whence $H \subset F$ and $\mathsf q (H) \subset \mathsf q (F)$ imply that $a \in F$.
If $a\in F^{\times}$, then $a=a^{N+1}(a^{N})^{-1}\in H$.  Suppose  $a\in F\setminus F^{\times}$. Since $[a]\in \mathcal C^*$ and $\mathcal C^*$ is finite, we obtain that $\langle[a]\rangle\subset \mathcal C^*$ is a finite cyclic group. If $N_0 \in \N$ is the order of $[a]$ in the cyclic group $\langle[a]\rangle \subset \mathcal C^*$,
then $a\sim a^{N_0t+1}$ for all $t\in \N$. Since $N_0N+1>N$, we obtain that $a^{N_0N+1}\in H$ and hence $a\in H$.

\smallskip
3. Suppose that $H$ is seminormal and set $\mathsf E ( \mathcal C^*) = \{e_0, \ldots, e_n\}$, where $n \in \N_0$,  $e_0= [1]$, and $e_0 + \ldots + e_n = e_n$. For $i \in [0,n]$, we set $\mathcal C^*_i = \mathcal C^*_{e_i}$. Since $\mathcal C^*$ is a union of groups,  \eqref{structure-Clifford} implies that
\[
\mathcal C^* = \biguplus_{i=0}^n \mathcal C^*_{i} \,.
\]
For every $i \in [0,n]$, we have $\mathcal C^*_i = \{g \in \mathcal C^* \mid mg = e_i \ \text{for some } \ m \in \N\}$ and $\mathcal C_i^* + C_n^* = C_n^*$.

For the class group of the Krull monoid $\widehat H$, we have
\[
\mathcal C (\widehat H) = \mathsf q (\mathcal F (P))/\mathsf q (H_0) = \mathsf q (F)/\mathsf q (\widehat H) = \mathsf q (F)/\mathsf q (H) \,.
\]
Since $H_0 \hookrightarrow \mathcal F (P)$ is a divisor theory, the inclusion $\widehat H = \widehat H^{\times} \times H_0 \subset F = \widehat H^{\times} \times \mathcal F (P)$ is saturated whence $\widehat H = \mathsf q (\widehat H) \cap F = \mathsf q (H) \cap F$ (\cite[Corollary 2.4.3.2]{Ge-HK06a}).

3.(a) Let $a\in H$. If $a\in H^{\times}$, then $[a]=[1]$ is an idempotent of $\mathcal C^*$. Suppose $a\in H\setminus H^{\times}$.
We have to show that $a\sim a^2$ and to do so we choose some
$b\in F$. If $ab\in H$, then obviously $a^2b\in H$. Conversely, suppose that $a^2b\in H$. Since $\mathcal C^*$ is a finite Clifford semigroup,   there exists an $m \in \N$ such that $a^{m+1}\sim a$.  Then  $ab\sim a^{m+1}b=a^{m-1}a^2b\in H$ which implies  $ab\in H$.

3.(b)  By \cite[Proposition 2.8.7.1]{Ge-HK06a},
the map $\Phi \colon \mathcal C^*\rightarrow \mathcal C (\widehat{H})$ given  by $\Phi([a])=a\mathsf q(\widehat{H})$ is an epimorphism.
Therefore $\Phi_n = \Phi |_{\mathcal C^*_n} \colon  \mathcal C^*_n\rightarrow\mathcal C(\widehat{H})$ is a group homomorphism, and it remains to show that it is bijective.

Let $b\in F\setminus F^{\times}$ such that $[b]=e_n$. Since $e_n$ is the identity of $\mathcal C_n^*$,  $\Phi_n (e_n)$ is the identity element of $\mathcal C(\widehat{H})$ whence $b \mathsf q (\widehat H) = \Phi_n (e_n) = \mathsf q (H)$ and  $b \in \mathsf q (H) \cap F = \widehat{H}$.
Then for every $a\in F$ we have
$[ba]=[b]+[a]\in \mathcal C^*_n$ and $\Phi([ba])=ba\mathsf q(\widehat{H})=a\mathsf q(\widehat{H})$. Thus $\Phi_{n}$ is surjective.

 Let $a\in F\setminus F^{\times}$  with $[a]\in \mathcal C^*_n$ such that $\Phi_n ([a])$ is the identity element of $\mathcal C(\widehat{H})$. Then $a \mathsf q (H) = \mathsf q (H)$ and $a \in \mathsf q (H) \cap F = \widehat{H}$ whence there exists $c\in H$ such that $ca\in H$.  Thus 3.(a) implies that  $[ca]$ and $[c]$ are  idempotents.
Since $[a]$ and $e_n$ are in $\mathcal C^*_n$ and $C_i^* + C_n^* = C_n^*$ for all $i \in [0,n]$, it follows that $[ca] = [c]+[a]$ and  $[c]+e_n$ are in $\mathcal C^*_n$. Since $[ca]=[ca]+[ca]$, it follows that $[ca]=e_n$ and similarly  $[c]+e_n=e_n$. Thus we obtain that
\[
[a]=[a]+e_n=[a]+[c]+e_n=[ca]+e_n=e_n \,,
\]
which implies that $\Phi_{n}$ is injective.
\end{proof}

\medskip
In our next remark and also in Section \ref{4}, we need the concept of monoids of product-one sequences over groups. Let $G$ be a finite group and $\mathcal F (G)$ be the free abelian monoid with basis $G$. An element $S = g_1  \bdot \ldots \bdot g_{\ell} \in \mathcal F (G)$ is said to be a {\it product-one sequence} (over $G$) if its terms can be ordered such that their product equals $1_G$, the identity element of the group $G$. The monoid $\mathcal B (G)$ of all product-one sequences over $G$ is a finitely generated C-monoid, and it is a Krull monoid if and only if $G$ is abelian (\cite[Theorem 3.2]{Cz-Do-Ge16a}).

\medskip
\begin{remark} \label{3.1}~

1. The main statement of Theorem \ref{1.1} was proved first for  the monoid of product-one sequences over a finite group. Let $G$ be a finite group and $G'$ be its commutator subgroup.  Jun Seok Oh showed that $\mathcal B (G)$ is seminormal if and only if $|G'|\le 2$  if and only if its class semigroup $\mathcal C^* ( \mathcal B (G), \mathcal F (G))$ is a union of groups (\cite[Corollary 3.12]{Oh19a}).

2. Let $H \subset F$ be a C-monoid with all conventions as in Theorem \ref{1.1}. If $a \in H$, then Theorem \ref{1.1}.1 implies that the element $[a] \in \mathcal C (H,F)$ is idempotent. In case $H = \mathcal B (G)$, the converse holds and this fact was used in the characterization when $\mathcal B (G)$ is seminormal. However, the converse does not hold true for general C-monoids as the next example shows.

3. Let $F= \mathcal F (P)$ with $P = \{p_1,p_2\}$ and
\[
H = [p_1p_2, p_1^{2k}p_2 \mid  k\ge 0]  =\{1, p_1p_2\}\cup \{p_1^{2k}p_2\mid k\ge 0\}\cup \{p_1^tp_2^s \mid  t\ge 0,  s\ge 2\} \subset F \,.
\]
Then $(H \DP p_1^2)=(H \DP p_1^4)=H\setminus\{1, p_1p_2\}$ which implies that $[p_1^2]=[p_1^4] \in \mathcal C (H,F)$.
Thus $[p_1^2]$ is an idempotent but $p_1^2 \not\in H$.
\end{remark}

\medskip
\section{A characterization of half-factoriality} \label{4}
\medskip

Let $H$ be a monoid and $\mathcal A (H)$ the set of atoms (irreducible elements) of $H$.  If an element $a \in H$  has a factorization of the form $a=u_1 \cdot \ldots \cdot u_k$, where $k \in \N$ and $u_1, \ldots, u_k \in \mathcal A (H)$, then $k$ is called a factorization length of $a$. The set $\mathsf L_H (a) = \mathsf L (a)$ of all factorization lengths is called the {\it set of lengths} of $a$, and for simplicity we set $\mathsf L (a) = \{0\}$ for $a \in H^{\times}$.  If $H$ is $v$-noetherian (which holds true of all C-monoids), then every $a \in H$ has a factorization into atoms and all sets of lengths are finite (\cite[Theorem 2.2.9]{Ge-HK06a}). For a finite set $L = \{m_0, \ldots, m_k\} \subset \Z$, where $k \in \N_0$ and $m_0 < \ldots < m_k$, $\Delta (L) = \{m_i - m_{i-1} \mid i \in [1,k]\}$ is called the set of distances of $L$. Then
\[
\mathcal L (H) = \{\mathsf L (a) \mid a \in H \} \quad \text{resp.} \quad \Delta (H) = \bigcup_{L \in \mathcal L (H)} \Delta (L)
\]
is called the {\it system of sets of lengths} resp. the {\it set of distances} of $H$.
The monoid $H$ is {\it half-factorial} if $\Delta (H)= \emptyset$ (equivalently, $|L|=1$ for all  $L \in \mathcal L (H)$).

Half-factoriality has always been a central topic in factorization theory (e.g., \cite{Ch-Co00, Co05a,  Co-Ma-Ok17a,  Ma-Ok16a,Ph12b, Sc05c}). In 1960 Carlitz proved that a ring of integers in an algebraic number field is half-factorial if and only if the class group has at most two elements. This result (which has a simple proof nowadays) carries over to Krull monoids. Indeed, a Krull monoid with class group $G$, that has a prime divisor in each class, is half-factorial if and only if $|G| \le 2$ (the assumption on the distribution of prime divisors is crucial; we refer to \cite[Section 5]{Gi06a} for background if this assumption fails). We will need this result and the involved machinery in our study of half-factoriality for C-monoids. Thus, we recall that  a monoid homomorphism $\theta \colon H \to B$ is  a {\it transfer homomorphism} if the following  conditions hold{\rm \,:}
\begin{enumerate}
\item[{\bf (T\,1)\,}] $B = \theta(H) B^\times$ \ and \ $\theta ^{-1} (B^\times) = H^\times$.

\item[{\bf (T\,2)\,}] If $u \in H$, \ $b,\,c \in B$ \ and \ $\theta (u) = bc$, then there exist \ $v,\,w \in H$ \ such that \ $u = vw$, \ $\theta (v)B^{\times} = b B^{\times}$ \ and \ $\theta (w)B^{\times} = cB^{\times}$.
\end{enumerate}
If $\theta \colon H \to B$ is a transfer homomorphism, then $\mathcal L (H) = \mathcal L (B)$ (\cite[Proposition 3.2.3]{Ge-HK06a}) whence $H$ is half-factorial if and only if $B$ is half-factorial.

All generalizations of Carlitz's result beyond the setting of Krull monoids have turned out to be surprisingly difficult. We mention two results valid for special classes of C-domains.  There is a characterization of half-factoriality for orders in quadratic number fields in number theoretic terms (\cite[Theorem 3.7.15]{Ge-HK06a}) and for a class of seminormal weakly Krull domains (including seminormal orders in number fields)  in algebraic terms involving the $v$-class group and extension properties of prime divisorial ideals (see \cite[Theorem 6.2]{Ge-Ka-Re15a} and \cite{Ge-Zh16c}).

In this section we establish a characterization of half-factoriality in terms of the class semigroup, that is valid for all seminormal C-monoids and based on Theorem \ref{1.1}. Although it is not difficult to show that the set of distances is finite for all C-monoids (\cite[Theorem 3.3.4]{Ge-HK06a}), a characterization of half-factoriality  turns out to be quite involved compared with the simpleness of the result for Krull monoids.  But this difference in complexity stems from the fact that the structure of the class semigroup of a C-monoid $H$ can be much more intricate than the structure of the class group $\mathcal C (\widehat H)$ of its complete integral closure. We provide an explicit example of a half-factorial seminormal C-monoid (Example \ref{4.3}) and we also refer to the explicit examples of class semigroups given in \cite[Section 4]{Oh18a}.

We  introduce our notation which remains valid throughout the rest of this section. Let $H\subset F=F^{\times}\times \mathcal F(P)$ be a dense seminormal C-monoid with $\mathsf E(\mathcal C)=\mathsf E (\mathcal C^*) = \{e_0 = [1], \ldots, e_n\}$, where $n \in \N$ and $e_0+ \ldots + e_n = e_n$. For a subset $T \subset F$, we define
\[
\mathcal C_T (H,F) = \{ [y] \mid y \in T\} \subset \mathcal C (H,F) \,.
\]
We use the abbreviations
\[
\begin{aligned}
\mathcal C = \mathcal C (H,F), \ \mathcal C^* = \mathcal C^* (H,F), \ \mathcal C_H = \mathcal C_H (H,F),  \
 \mathcal C_i^*=\mathcal C^*_{e_i}, \quad \text{and} \quad \mathcal C_i = \mathcal C_{e_i} \quad \text{for all} \  i \in [0,n]  \,,
\end{aligned}
\]
whence
\[
\mathcal C  = \biguplus_{i=0}^n \mathcal C_i, \   \mathcal C^*=\biguplus_{i=0}^n \mathcal C_i^*,\  \
\mathcal C_i=\mathcal C_i^* \ \text{for all} \  i \in [1,n],\  \text{ and }\ \mathcal C^{\times}=\mathcal C_0=\mathcal C_{F^{\times}}(H,F)\cup \mathcal C_0^*,
\]

\medskip
\begin{proposition} \label{4.1}
Let  $H\subset F=F^{\times}\times \mathcal F(P)$ be a dense seminormal   \C-monoid $($with all notations  as above$)$ and  suppose  that every class of  $\mathcal C^*$ contains an element from $P$.
For every $i\in [0,n]$, let $P_i=\{p\in P\mid [p]\in \mathcal C_i^*\}$, $F_i=F^{\times}\times \mathcal F(P_i)$,  $H_i=F_i\cap H$, and let $\varphi_i \colon  \mathcal C_{F^{\times}}(H,F)\rightarrow \mathcal C_i^*$ be defined by $\varphi_i([\epsilon])=[\epsilon]+e_i$ for all $\epsilon\in F^{\times}$.
Then for every  $k\in [0,n]$, $H_k \subset F_k$ is a seminormal \C-monoid and the following statements hold{\rm \,:}

\begin{enumerate}
\item If $e_k\not\in \mathcal C_H$, then $H_k=H^{\times}$.  If $e_k\in \mathcal C_H$, then
      \[
      \mathcal C(H_k, F_k) \cong
      \begin{cases}
      \mathcal C_k,  & \text{if $\varphi_k$ is injective}, \\
      \mathcal C_{F^{\times}}(H,F) \cup \mathcal C_k,  & \text{if  $\varphi_k$ is not injective.}
      \end{cases}
      \]
      In particular, $\mathcal C(H_0,F_0)\cong \mathcal C_0$ and if $\mathcal C (H_k, F_k)$ is a group, then $H_k$ is a Krull monoid.

 \item  If $e_k\in \mathcal C_H$, then there is a transfer homomorphism $\theta_k \colon H_k\rightarrow \mathcal B \big(\mathcal C_k^*/\varphi_k( \mathcal C_{F^{\times}}(H,F)) \big)$.
\end{enumerate}

\end{proposition}

\begin{proof}
Let $k\in [0,n]$. By \cite[Proposition 2.9.9]{Ge-HK06a}, $H_k \subset F_k$ is a C-monoid and it is a divisor-closed submonoid of $H$. Since $H_k$ is a divisor-closed submonoid of the seminormal monoid $H$, it  is seminormal (this is easy to check; for details see \cite[Lemma 3.2]{Ge-Ka-Re15a}).
By \cite[Lemma 2.8.4.5]{Ge-HK06a}, $\psi_k \colon \mathcal C_{F_k}(H,F)\rightarrow \mathcal C(H_k,F_k)$, defined by $\psi_k([a]_H^F)=[a]_{H_k}^{F_k}$ for  $a\in F_k$, is an epimorphism. By definition, we have $\mathcal C_{F_k}(H,F)= \mathcal C_{F^{\times}}(H,F)\cup \mathcal C_k^*=\mathcal C_{F^{\times}}(H,F)\cup \mathcal C_k$.

1. First suppose $e_k\not\in \mathcal C_H$. It is clear that $H^{\times}\subset H\cap F_k=H_k$. Assume to the contrary that there exists an element $a\in H_k\setminus H^{\times}$. If $a\in F^{\times}$, then Theorem \ref{1.1}.1 implies that $[a]_H^F=[1]_H^F$ which implies $a\in H^{\times}$, a contradiction. If $a\in F_k\setminus F^{\times}$, then $[a]_H^F\in \mathcal C_k$ and again Theorem \ref{1.1}.1 implies that $[a]_H^F=e_k\in\mathcal C_H$, a contradiction.

Now we suppose that $e_k \in \mathcal C_H$.

1.(a) Suppose that $\varphi_k$ is injective. We want to show that $\psi_{k}|_{\mathcal C_k} \colon \mathcal C_k\rightarrow \mathcal C(H_k,F_k)$ is  bijective.

 First, we show that $\psi_{k}|_{\mathcal C_k} $ is surjective.  Let $p\in P_k$ such that $[p]_H^F=e_k$ and $\epsilon\in F^{\times}$. Then $p\in H_k$. It suffices to prove $[\epsilon]_{H_k}^{F_k}=[\epsilon p]_{H_k}^{F_k}$.

Let $a\in F_k$. If $\epsilon a\in H_k$, then $\epsilon pa\in H_k$. Suppose $\epsilon pa\in H_k$. Then $[\epsilon pa]_H^F=[\epsilon a]_H^F+e_k=e_k$. If $[\epsilon a]_H^F\in \mathcal C_k$, then $[\epsilon a]_H^F=e_k$ and hence $\epsilon a\in H\cap F_k= H_k$. Suppose
$[\epsilon a]_H^F\in  \mathcal C_{F^{\times}}(H,F)$.  Since $\varphi_k$ is injective, we obtain  $[\epsilon a]_H^F=[1]_H^F$ whence $\epsilon a\in H_k$.

In order to  show that  $\psi_{k}|_{\mathcal C_k}$ is injective, let
$p_1, p_2\in P_k$ be given such that $[p_1]_H^F\neq [p_2]_H^F$. Since every class of $\mathcal C^*$ contains an element from $P$, there is a  $p\in P_k$ such that $[p]_H^F=-[p_1]_H^F$. Then $pp_1\in H_k$ and $pp_2\not\in H_k$ which implies  that $[p_1]_{H_k}^{F_k}\neq [p_2]_{H_k}^{F_k}$.

\smallskip
1.(b) Suppose that $\varphi_k$ is not injective. It suffices to prove $\psi_k$ is injective.
Let $a_1, a_2\in F_k$ such that $[a_1]_H^F\neq [a_2]_H^F$. We have to show that $[a_1]_{H_k}^{F_k}\neq [a_2]_{H_k}^{F_k}$ and to do so we distinguish three cases.

 If $[a_1]_H^F, [a_2]_H^F\in  \mathcal C_{F^{\times}}(H,F)$, then there exists an $\epsilon \in  F^{\times}$ such that $[\epsilon]_H^F=-[a_1]_H^F$. Therefore $\epsilon a_1\in H_k$ and $\epsilon a_2\not\in H_k$ which imply $[a_1]_{H_k}^{F_k}\neq [a_2]_{H_k}^{F_k}$.

If $[a_1]_H^F, [a_2]_H^F\in \mathcal C_k$, then there exists an $a \in F_k$ such that $[a]_H^F=-[a_1]_H^F$. Therefore $a a_1\in H_k$ and $a a_2\not\in H_k$ which imply $[a_1]_{H_k}^{F_k}\neq [a_2]_{H_k}^{F_k}$.

By symmetry, it remains to consider the case where $[a_1]_H^F \in \mathcal C_k$ and $ [a_2]_H^F\in  \mathcal C_{F^{\times}}(H,F)$. Then there exists $\epsilon \in  F^{\times}$ such that $[\epsilon]_H^F=-[a_2]_H^F$. Therefore $\epsilon a_2\in H_k$. Assume to the contrary that $[a_1]_{H_k}^{F_k}=[a_2]_{H_k}^{F_k}$. Then $\epsilon a_1\in H_k$. Since $\varphi_k$ is not injective, there exists  $\eta\in F^{\times}$ such that $[\eta]_H^F\neq [1]_H^F$ and $[\eta]_H^F+e_k=e_k$. Then $[\eta \epsilon a_1]=e_k$ whence $\eta\epsilon a_1\in H_k$ and $\eta \epsilon a_2\in H_k$. But $[\eta\epsilon a_2]_H^F=[\eta]_H^F\neq [1]_H^F$, a contradiction.

\smallskip
1.(c) Since $\mathcal C_{F^{\times}}(H,F)\cup \mathcal C_0=\mathcal C_0$, we have $\mathcal C(H_0,F_0)\cong \mathcal C_0$. If $\mathcal C (H_k, F_k)$ is a group, then $H_k$ is a Krull monoid by \eqref{group} and \eqref{Krull}.

\medskip
2.  Suppose that $e_k \in \mathcal C_H$. For $g\in \mathcal C_k^*$, we set $\overline{g}=g+\varphi_k( \mathcal C_{F^{\times}}(H,F))\in \mathcal C_k^*/\varphi_k( \mathcal C_{F^{\times}}(H,F))$ and we define
\[
\begin{aligned}
\theta^* \colon F_k & \ \longrightarrow \ \mathcal F(\mathcal C_k^*/\varphi_k( \mathcal C_{F^{\times}}(H,F))) \\ a=\epsilon \prod_{p\in P_k}p^{\mathsf v_p(a)} & \ \longmapsto \ \prod_{p\in P_k}\overline{[p]_H^F}^{\mathsf v_p(a)} \,.
\end{aligned}
\]
First, we show that $\theta^*(H_k)=\mathcal B(\mathcal C_k^*/\varphi_k( \mathcal C_{F^{\times}}(H,F)))$.
 Note $H_k=H\cap F_k=\big\{x\in F_k\mid [x]_H^F\in \{e_0, e_k\}\big\}$. If $a\in H_k^{\times}$, then $\theta^*(a)=1_{\mathcal F(\mathcal C_k^*/\varphi_k( \mathcal C_{F^{\times}}(H,F)))}$.
If $a=\epsilon \prod_{p\in P_k}p^{\mathsf v_p(a)}\in H_k\setminus H_k^{\times}$, then the sum of the elements of  $\theta^*(a)$ equals $\sum_{p\in P_k}\mathsf v_p(a)\overline{[p]_H^F} = \overline{[a]_H^F}=\overline{e_k}$, which is the zero element of the group $\mathcal C_k^*/\varphi_k( \mathcal C_{F^{\times}}(H,F))$. Thus we obtain that $\theta^*(H_k)\subset \mathcal B(\mathcal C_k^*/\varphi_k( \mathcal C_{F^{\times}}(H,F)))$ and in order to verify equality, we choose an $S=\overline{g_1} \cdot \ldots  \cdot \overline{g_{\ell}}\in \mathcal B(\mathcal C_k^*/\varphi_k( \mathcal C_{F^{\times}}(H,F)))$, where $g_1,\ldots g_{\ell}\in \mathcal C_k^*$. By assumption, there exist  $p_1,\ldots ,p_{\ell}\in P_k$ such that $[p_i]_H^F=g_i$ for all $i\in [1,\ell]$.
Since
\[
\overline{[p_1 \cdot \ldots \cdot  p_{\ell}]_H^F} = \overline{[p_1]_H^F} + \ldots + \overline{[p_{\ell}]_H^F} =
\overline{g_1} + \ldots  + \overline{g_{\ell}} = 0_{\mathcal C_k^*/\varphi_k( \mathcal C_{F^{\times}}(H,F))} \,,
\]
we infer that $[p_1 \cdot \ldots \cdot  p_{\ell}]_H^F\in \varphi_k( \mathcal C_{F^{\times}}(H,F))$. Thus there is an $\epsilon\in F^{\times}$ such that $\varphi_k([\epsilon]_H^F)=-[p_1 \cdot \ldots \cdot  p_{\ell}]_H^F$. So it follows that $[\epsilon p_1 \cdot \ldots \cdot p_{\ell}]_H^F = e_k$ whence $\epsilon p_1 \cdot \ldots \cdot  p_{\ell}\in H\cap F_k=H_k$ and $\theta^*(\epsilon p_1 \cdot \ldots \cdot p_{\ell})=S$.

Secondly, we show $\theta=\theta^*|_{H_k}\colon H_k\rightarrow \mathcal B(\mathcal C_k^*/\varphi_k( \mathcal C_{F^{\times}}(H,F)))$ is a transfer homomorphism. Clearly, $\mathcal B (\cdot)$ is reduced and {\bf (T1)} holds. In order to verify {\bf (T2)}, let  $a\in H_k$ and $b,c\in \mathcal B(\mathcal C_k^*/\varphi_k( \mathcal C_{F^{\times}}(H,F)))$ such that $\theta(a)=bc$.
We set $a=\epsilon p_1 \cdot \ldots \cdot p_{\ell}$, where $\epsilon\in F^{\times}$, $\ell \in \N_0$, $p_1,\ldots, p_{\ell}\in P_k$ and, after renumbering if necessary, we assume $b=\overline{[p_1]_H^F} \cdot \ldots \cdot \overline{[p_s]_H^F}$ and $c=\overline{[p_{s+1}]_H^F} \cdot \ldots \cdot \overline{[p_{\ell}]_H^F}$, where $s\in [0,\ell]$. Thus there exists an $\epsilon_1\in F^{\times}$ such that $[\epsilon_1]_H^F=-[p_1 \cdot \ldots \cdot  p_s]_H^F$ and hence  $$[\epsilon\epsilon_1^{-1}]_H^F=[\epsilon]_H^F-[\epsilon_1]_H^F=-[p_1 \cdot \ldots \cdot p_{\ell}]_H^F+[p_1 \cdot \ldots \cdot p_s]_H^F=-[p_{s+1} \cdot \ldots \cdot p_{\ell}]_H^F\,.$$
It follows that $u=\epsilon_1 p_1 \cdot \ldots \cdot p_s\in H_k$, $v=\epsilon\epsilon_1^{-1}p_{s+1} \cdot \ldots \cdot p_{\ell}\in H_k$, $a=uv$,  $\theta(u)=b$, and $\theta(v)=c$.
\end{proof}

\begin{theorem} \label{4.2}
Let $H\subset F=F^{\times}\times \mathcal F(P)$ be a  dense seminormal \C-monoid $($with all notations as above$)$ and suppose  that every class of  $\mathcal C^*$ contains an element from $P$.
For every $i\in [0,n]$, let $\varphi_i \colon  \mathcal C_{F^{\times}}(H,F)\rightarrow \mathcal C_i^*$ be defined by $\varphi_i([\epsilon])=[\epsilon]+e_i$ for all $\epsilon\in F^{\times}$ and set $\mathcal C_i'=\mathcal C_i^*/\varphi_i( \mathcal C_{F^{\times}}(H,F))$.
 For distinct $i,j\in [0,n]$, let $\phi_{i,j} \colon \mathcal C_i^*\rightarrow \mathcal C_i^*+\mathcal C_j^*$  be defined by $\phi_{i,j}(g_i)=g_i+e_j$ for all $g_i\in \mathcal C_i^*$.

\noindent
Then $H$ is half-factorial if and only if the following properties hold{\rm \,:}
\begin{enumerate}
\item[P1.]  $|\mathcal C_n'|\le 2$.

\item[P2.]  For every $i\in [0,n]$, we have
            \[
            \phi_{i,n}^{-1}(\varphi_n( \mathcal C_{F^{\times}}(H,F)))=\left\{
            \begin{aligned}
            &\varphi_i( \mathcal C_{F^{\times}}(H,F)),& \quad\quad& e_i\in \mathcal C_H\\
            &\mathcal C_i,&& e_i\not\in \mathcal C_H
            \end{aligned}
            \right.
            \]

\item[P3.] For distinct $i,j\in [0,n]$, if $e_i,e_j\in \mathcal C_H$, then
$\ker (\phi_{i,j} \circ \varphi_i)=\ker(\varphi_i)+ \ker (\varphi_j)$.

\item[P4.] $\mathsf E(\mathcal C^*)\setminus \mathcal C_H$ is additively closed and for distinct $i_1,i_2, j\in [0,n]$, if $e_{i_1}, e_{i_2}\in \mathcal C_H$ and $e_j\in \mathsf E(\mathcal C^*)\setminus \mathcal C_H$ such that $e_{i_1}+e_{i_2}+e_j\in \mathcal C_H$, then $e_{i_1}+e_j\in \mathcal C_H$ or $e_{i_2}+e_j\in \mathcal C_H$.
\end{enumerate}
\end{theorem}

\begin{proof}
In order to show that $e_n\in \mathcal C_H$, we choose a
$p^* \in P$ such that $[p^*]=e_n$. Since $H\subset F $ is dense, there exists an $a\in H$ such that $p^* \mid a$, say $a=p^*b$ with $b \in F$. By Theorem \ref{1.1}, $[a]$ is an idempotent and since $[a] = [p^*]+[b]=e_n+[b] \in \mathcal C_n$, it follows that  $e_n  = [a] \in \mathcal C_H$.

\smallskip
\noindent
{\bf 1.} We suppose that  $H$ is half-factorial and verify properties (P1) - (P4).

\smallskip
(P1). Let $P_n=\{p\in P\mid [p]\in \mathcal C_n^*\}$, $F_n=F^{\times}\times \mathcal F(P_n)$, and $H_n=H\cap F_n$. Since $H_n$ is a divisor closed submonoid of $H$, $H_n$ is half-factorial. By Proposition \ref{4.1}.3, there is a transfer homomorphism $\theta \colon H_n \to \mathcal B (C_n')$. Thus $\mathcal B (C_n')$ is half-factorial which implies  that $|\mathcal C_n'|\le 2$ (\cite[Corollary 3.4.12]{Ge-HK06a}).

\smallskip
(P2). Let $i\in [0,n]$ and $e_i\in \mathcal C_H$. Clearly, $ e_i+e_n=e_n$ and $\phi_{i,n}( \mathcal C_{F^{\times}}(H,F)+e_i)= \mathcal C_{F^{\times}}(H,F)+e_n$.
 Assume to the contrary that  there exists $g_i\in \mathcal C_i$ with $g_i\not\in \varphi_i( \mathcal C_{F^{\times}}(H,F))$ such that $\phi_{i,n}(g_i)\in  \mathcal C_{F^{\times}}(H,F)+e_n$, i.e. there exists $\epsilon\in  F^{\times}$ such that  $g_i+e_n=[\epsilon]+e_n$. Let $p_1, p_2, q\in P$ such that $[p_1]=g_i$, $[p_2]=-g_i$, and $[q]=e_n$. Then $[\epsilon^{-1}p_1q]=[\epsilon p_2q]=[q]=e_n$ and $[p_1p_2]=e_i$. We claim  that  $\epsilon^{-1}p_1q, \epsilon p_2q, p_1p_2, q \in \mathcal A (H)$. Clearly, the elements lie in $H$ and, for example, assume to the contrary that $p_1p_2$ is not an atom. Then there is an $\eta \in F^{\times}$ such that $\eta p_1, \eta^{-1}p_2 \in H$. Then $[\eta p_1] = [\eta] + [p_1] = e_i$ whence $g_i = [p_1] = [\eta^{-1}]+e_i \in \varphi_i \big( \mathcal C_{F^{\times}} (H,F) \big)$, a contradiction.
 Thus we obtain that
$(p_1p_2)(q)^{2}=(\epsilon ^{-1}p_1q)(\epsilon p_2q)$, a contradiction to the half-factoriality of $H$.

Suppose $e_i\in \mathsf E(\mathcal C^*)\setminus \mathcal C_H$. Assume to the contrary that there exists $g_i\in \mathcal C_i$ such that $\phi_{i,n}(g_i)=g_i+e_n\not\in  \mathcal C_{F^{\times}}(H,F)+ e_n$.
Let $p_1,p_2,p_3\in P$ such that $[p_1]=g_i$, $[p_2]=e_n$, and $[p_3]=-g_i+e_n$, and let $\ord (g_i) = \ord_{\mathcal C_i} (g_i)$ denote the order of $g_i$ in the group $\mathcal C_i$. Then  $[p_1^{\ord(g_i)}p_2]=[p_2]=[p_1p_3]=e_n\in \mathcal C_H$ which implies that $p_1^{\ord(g_i)}p_2, p_2, p_1p_3 \in \mathcal A (H)$. Therefore $(p_1p_3)^{\ord(g_i)}p_2=(p_1^{\ord(g_i)}p_2)(p_3^{\ord(g_i)})$. Since $p_3^{\ord(g_i)}\in H$ and for any $\epsilon \in F^{\times}$, $\epsilon p_3\not\in H$, we obtain $\ord(g_i)\not\in \mathsf L_H (p_3^{\ord(g_i)})$,  a contradiction to the half-factoriality of $H$.

\smallskip
(P3). Let $i,j\in [0,n]$ be distinct and $e_i,e_j\in \mathcal C_H$.  Since $ \phi_{i,j} \circ \varphi_i =\phi_{j,i} \circ \varphi_j $, we have $\ker ( \phi_{i,j} \circ \varphi_i) \supset\ker(\varphi_i)$ and $\ker (\phi_{i,j} \circ \varphi_i)\supset \ker (\varphi_j)$ which imply  $\ker (\phi_{i,j} \circ \varphi_i )\supset \ker(\varphi_i)+\ker (\varphi_j)$. Let $\epsilon\in F^{\times}$ such that $\phi_{i,j}(\varphi_i([\epsilon]))=e_i+e_j$. If $[\epsilon]+e_i=e_i$
or $[\epsilon]+e_j=e_j$, then $[\epsilon] = [\epsilon]+[1] \in \ker (\varphi_i)+\ker (\varphi_j)$. Suppose $[\epsilon]+e_i\neq e_i$ and $[\epsilon]+e_j\neq e_j$.
Let $p_1, p_2, p_3\in P$ such that $[p_1]=[\epsilon]+e_i$, $[p_2]=e_j$, and $[p_3]=e_i+e_j$.
Since $p_2, p_3, \epsilon p_3, \epsilon^{-1}p_1 \in \mathcal A (H)$, the equation $(p_1p_2)p_3=(\epsilon^{-1}p_1)(\epsilon p_3)p_2$ implies that $p_1p_2 \notin \mathcal A (H)$. Then there exists $\delta\in F^{\times}$ such that $\delta p_1 \in H$ and $\delta^{-1}p_2 \in H$ whence $[\delta p_1]= [\delta]+[\epsilon]+e_i=e_i$ and $[\delta^{-1}p_2]= -[\delta]+e_j=e_j$.
It follows that $[\delta]+[\epsilon]\in \ker(\varphi_i)$, $-[\delta]\in \ker(\varphi_j)$, and $[\epsilon]=[\delta]+[\epsilon]-[\delta]\in \ker(\varphi_i)+\ker(\varphi_j)$.

\smallskip
(P4). Assume to that contrary that there exist  $f_1,f_2 \in \mathsf E(\mathcal C^*)\setminus \mathcal C_H$ such that $f_1+f_2\in \mathcal C_H$. Let $q_1,q_2\in P$ such that $[q_1]=f_1$ and $[q_2]=f_2$. Then $[q_1q_2]=[q_1^2q_2]=[q_1q_2^2]=f_1+f_2\in \mathcal C_H$, $q_1q_2, q_1^2q_2, q_1q_2^2\in \mathcal A (H)$, and $(q_1q_2)^3= (q_1^2q_2) (q_1q_2^2)$, a contradiction to the half-factoriality of $H$.

Let $i_1,i_2, j\in [0,n]$ such that $e_j\in \mathsf E(\mathcal C^*)\setminus \mathcal C_H$ and $ e_{i_1}, e_{i_2}, e_{i_1}+e_{i_2}+e_j\in \mathcal C_H$. Assume to the contrary that $e_{i_1}+e_j, e_{i_2}+e_j\in \mathsf E(\mathcal C^*)\setminus \mathcal C_H$.
 Let $p,p_1,p_2,p_3\in P$ such that $[p]=e_{i_1}+e_{i_2},[p_1]=e_{i_1}, [p_2]=e_{i_2}, [p_3]=e_j$.
Then  $p, p_1, p_2,pp_3, p_1p_2p_3 \in \mathcal A (H)$  and  $(p)(p_1p_2p_3)=(p_1)(p_2)(pp_3)$,  a contradiction to the half-factoriality of $H$.

\medskip
{\bf  2.} We suppose that (P1) - (P4) hold and prove that  $H$ is half-factorial. We start with the following three assertions.

\begin{enumerate}
\item[{\bf A1.}\,] If $i \in [0,n]$ and $e_i\in \mathcal C_H$, then  $|\mathcal C_i'|\le 2$.

\smallskip

\item[{\bf A2.}\,] If $i,j \in [0,n]$ and $e_i,e_j\in \mathcal C_H$, then  $\phi_{i,j}^{-1}( \mathcal C_{F^{\times}}(H,F)+e_i+e_j)= \mathcal C_{F^{\times}}(H,F)+e_i$.

\smallskip

\item[{\bf A3.}\,] If $i,j \in [0,n]$, $e_i\not\in \mathcal C_H$, and $e_j\in \mathcal C_H$ such that $e_i+e_j\in \mathcal C_H$, then $\phi_{i,j}^{-1}( \mathcal C_{F^{\times}}(H,F)+e_i+e_j)=\mathcal C_i$.
\end{enumerate}

\smallskip

{\it Proof of \,{\bf A1}}.\,
Let $i \in [0,n]$ and  $e_i\in \mathcal C_H$. By (P2), there is a monomorphism $\phi_{i,n}' \colon \mathcal C_i'\rightarrow \mathcal C_n'$ defined by $\phi_{i,n}'(h+\varphi_i( \mathcal C_{F^{\times}}(H,F)))=h+e_n+\varphi_n( \mathcal C_{F^{\times}}(H,F))$ for all $h\in \mathcal C_i^*$. It follows by (P1) that $|\mathcal C_i'|\le 2$.

{\it Proof of \,{\bf A2}}.\, Let $i,j \in [0,n]$ and  $e_i,e_j\in \mathcal C_H$, say  $e_i+e_j=e_k$ with $k \in [0,n]$, and note that $\phi_{i,n}=\phi_{k,n} \circ \phi_{i,k}$ and $\phi_{i,j}=\phi_{i,k}$.
Then (P2) implies that $ \mathcal C_{F^{\times}}(H,F)+e_i=\phi_{i,n}^{-1}( \mathcal C_{F^{\times}}(H,F)+e_n)=\phi_{i,k}^{-1}(\phi_{k,n}^{-1}( \mathcal C_{F^{\times}}(H,F)+e_n))=\phi_{i,j}^{-1}( \mathcal C_{F^{\times}}(H,F)+e_k)$.

{\it Proof of \,{\bf A3}}.\, Let $i,j \in [0,n]$, $e_i\in \mathsf E(\mathcal C^*)\setminus \mathcal C_H$,  and $e_j\in \mathcal C_H$ such that $e_i+e_j \in \mathcal C_H$, say $e_i+e_j=e_k$ with $k \in [0,n]$, and again we note  that $\phi_{i,n}=\phi_{k,n} \circ \phi_{i,k}$ and $\phi_{i,j}=\phi_{i,k}$.
Then (P2) implies that $\mathcal C_i=\phi_{i,n}^{-1}( \mathcal C_{F^{\times}}(H,F)+e_n)=\phi_{i,k}^{-1}(\phi_{k,n}^{-1}( \mathcal C_{F^{\times}}(H,F)+e_n))=\phi_{i,j}^{-1}( \mathcal C_{F^{\times}}(H,F)+e_k)$.

\smallskip
Now we set $I=\{i\in [0,n]\mid e_i\in \mathcal C_H\}$ and we define
\[
G^1=\bigcup_{i\in I} \varphi_i \big( \mathcal C_{F^{\times}}(H,F) \big), \ G^2=\bigcup_{i\in I} \big(\mathcal C_i^*\setminus \varphi_i( \mathcal C_{F^{\times}}(H,F)) \big), \ \text{ and } \ G^3=\mathcal C^*\setminus \bigcup_{i\in I}\mathcal C_i^* \,.
\]
For every $a=\epsilon p_1 \cdot \ldots \cdot p_{\ell}\in F$, where  $\epsilon \in F^{\times}$, $\ell \in \N_0$, and  $p_1,\ldots, p_{\ell}\in P$, we define
\[
\mathsf l_1(a) = |\{j\in[1,\ell]\mid [p_j]\in G^1\}|\,,\
\mathsf l_2(a) = |\{k\in [1,\ell]\mid [p_k]\in G^2\}|\,, \quad
\text{and } \quad  \mathsf l(a)=\mathsf l_1(a)+\frac{1}{2}\mathsf l_2(a)\,.
\]
In order to prove that $H$ is half-factorial, it is sufficient to prove the following assertion.

\medskip
\noindent{\bf  A4. }{\it If $a$ is an atom of $H$, then $\mathsf l(a)=1$.}
\medskip

{\it Proof of \,{\bf A4}}.\,
Let $a=\epsilon p_1 \cdot \ldots \cdot p_{\ell}\in F\setminus F^{\times}$ be an atom of $H$. Assume to the contrary that $\mathsf l(a)=0$. Then,  for every $i\in [1,\ell]$, we have  $[p_i]\in G^3$, say $[p_i] \in \mathcal C_{f_i}$ with  $f_i\in \mathsf E(\mathcal C^*)\setminus \mathcal C_H$.
Clearly, we have
$[a]=[\epsilon]+[p_1]+\ldots+[p_{\ell}] \in \mathcal C_{[a]}$ and $[a]\in \mathcal C_{F^{\times}}(H, F) + \mathcal C_{f_1} + \ldots + \mathcal C_{f_{\ell}} \subset \mathcal C_{f_1+\ldots +f_{\ell}}$ whence $f_1+\ldots +f_{\ell}=[a]$, a contradiction to (P4).

Assume to the contrary that  $\mathsf l(a)\ge 2$ and distinguish three cases.

\smallskip
\noindent
CASE 1: \ $\mathsf l_1(a)\ge 2$.

After renumbering if necessary, we may assume that $[p_1],[p_2]\in G^1$. Then there exists $\epsilon_1,\epsilon_2\in F^{\times}$ and $e_i,e_j\in \mathcal C_H$ such that $[p_1]=[\epsilon_1]+e_i, [p_2]=[\epsilon_2]+e_j$. We set  $g=[\epsilon p_3 \cdot \ldots \cdot p_{\ell}]\in\mathcal C_k$ and $[a]=e_{r}\in \mathcal C_H$ for some $k,r\in [0,n]$.
Then
\[
e_r = [\epsilon_1] + e_i + [\epsilon_2]+e_j + g \in \mathcal C_{F^{\times}} (H,F)+ \mathcal C_i + \mathcal C_j + \mathcal C_k \subset \mathcal C_{e_i+e_j+e_k}
\]
whence
$e_i+e_j+e_k=e_{r}\in \mathcal C_H$. By (P4), we obtain that $e_i+e_k\in \mathcal C_H$ or $e_j+e_k\in \mathcal C_H$, say $e_j+e_k=e_t\in \mathcal C_H$ for some $t\in [0,n]$.

Since   $[\epsilon_1]+[p_2]+g\in \mathcal C_t$ and
\[
\phi_{t,i}([\epsilon_1]+[p_2]+g)= [\epsilon_1] + [p_2]+g+e_i = [p_1]+[p_2]+g= [\epsilon p_1 \cdot \ldots \cdot p_{\ell}]=e_r = e_i+e_t \in  \mathcal C_{F^{\times}}(H,F)+e_i+e_t,
\]
we obtain $[\epsilon_1]+[p_2]+g\in  \mathcal C_{F^{\times}}(H,F)+e_t$ by  {\bf A2}. It follows that  $[\epsilon p_2 \cdot \ldots \cdot p_{\ell}]=[p_2]+g\in  \mathcal C_{F^{\times}}(H,F)+e_t$ whence there exists $\epsilon_3\in F^{\times}$ such that $[\epsilon p_2 \cdot \ldots \cdot p_{\ell}]=[\epsilon_3]+e_t$.
Since
\[
e_i+e_t=e_r = [a]=[p_1]+[\epsilon p_2 \cdot \ldots \cdot p_{\ell}]= [\epsilon_1]+e_i+[\epsilon_3]+e_t \,,
\]
we infer that  $[\epsilon_1]+[\epsilon_3]\in \ker(\phi_{i,t} \circ \varphi_i )$. By (P3), there exists $\delta\in F^{\times}$ such that $[\delta]\in \ker(\varphi_i)$ and $[\epsilon_1\epsilon_3\delta^{-1}]\in \ker(\varphi_t)$.
Thus we obtain that $$[\delta \epsilon_1^{-1}p_1]=[\delta]-[\epsilon_1]+[p_1]=[\delta]+e_i=e_i$$
 $$[\epsilon_1\delta^{-1}\epsilon p_2 \cdot \ldots \cdot p_{\ell}]=[\epsilon_1\epsilon_3\delta^{-1} \epsilon_3^{-1} \epsilon p_2 \cdot \ldots \cdot p_{\ell}]=[\epsilon_1\epsilon_3\delta^{-1}]-[\epsilon_3]+[\epsilon p_2 \cdot \ldots \cdot p_{\ell}]=[\epsilon_1\epsilon_2\delta^{-1}]+e_t=e_t\,.$$
Thus $\delta \epsilon_1^{-1}p_1 \in H \setminus H^{\times}$, $\epsilon_1\delta^{-1}\epsilon p_2 \cdot \ldots  \cdot p_{\ell} \in H \setminus H^{\times}$, and   $a=(\delta \epsilon_1^{-1}p_1)(\epsilon_1\delta^{-1}\epsilon p_2 \cdot \ldots  \cdot p_{\ell})$, a contradiction to $a \in \mathcal A (H)$.

\smallskip
\noindent
CASE 2: \ $\mathsf l_1(a)=1$ and $\mathsf l_2(a)\ge 2$.

After renumbering if necessary, we may assume that $[p_1]\in G^1$, and $[p_2], [p_3]\in G^2$. There are $i,j, k\in [0,n]$ such that $[p_2]\in \mathcal C_i^* \setminus \varphi_i (\mathcal C_{F^{\times}} (H,F)$,  $[p_3]\in \mathcal C_j^* \setminus \varphi_j (\mathcal C_{F^{\times}} (H,F)$, and  $e_k=e_i+e_j$.  By {\bf A2} we infer that
\[
\phi_{i,k}([p_2])=[p_2]+e_k\not\in  \mathcal C_{F^{\times}}(H,F)+e_i+e_k=\varphi_k( \mathcal C_{F^{\times}}(H,F))
\]
 and
\[
 \phi_{i,k}([p_3])=[p_3]+e_k\not\in  \mathcal C_{F^{\times}}(H,F)+e_i+e_k=\varphi_k( \mathcal C_{F^{\times}}(H,F)) \,.
\]

Since $[p_2]+e_k$ and $[p_3]+e_k$ are in $\mathcal C_k^* \setminus \varphi_k (\mathcal C_{F^{\times}} (H,F))$ and since, by {\bf A1},
$|\mathcal C_k'| \le 2$, it follows that $[p_2]+[p_3]= ([p_2]+e_k)+([p_3]+e_k) \in \varphi_k (\mathcal C_{F^{\times}} (H,F)) \subset G^1$.
We choose a   $q\in P$ such that $[q]=[p_2]+[p_3]$. Then $b=\epsilon p_1qp_4 \cdot \ldots \cdot p_{\ell} \in \mathcal A (H)$ with  $\mathsf l(b)=\mathsf l(a)\ge 2$. Now we are back to CASE 1.

\smallskip
\noindent
CASE 3: \ $\mathsf l_1(a)=0$ and $\mathsf l_2(a)\ge 4$.

After renumbering if necessary, we may assume that $[p_1], [p_2], [p_3], [p_4]\in G^2$.
Arguing  as in CASE 2 we infer that  $[p_1]+[p_2], [p_3]+[p_4]\in G^1$. Let $q_1,q_2\in P$ such that $[q_1]=[p_1]+[p_2]$ and $[q_2]=[p_3]+[p_4]$. Then $b=\epsilon q_1q_2p_4 \cdot \ldots \cdot p_{\ell} \in \mathcal A (H)$ with $\mathsf l(b)=\mathsf l(a)\ge 2$ whence we are back to CASE 1.
\end{proof}

\smallskip
The following example shows that both, the number and the size, of the constituent groups of the reduced class semigroup of a half-factorial seminormal \C-monoid can be arbitrarily large.

\medskip
\begin{example} \label{4.3}
Let $C$ be a Clifford semigroup with
\[
\mathsf E(C)=\{e_0,e_1,\ldots, e_n\}, \ C=\biguplus_{i=0}^n C_i , \quad \text{ and } \quad C_{e_i} = C_i = G_i\oplus \Z/2 \Z \,,
\]
where $n \in \N$, $i,j  \in [0,n]$, $C_i+e_j=C_j$ whenever $i<j$,  and  $G_0\supsetneq \ldots \supsetneq G_n=\{1\}$ are abelian groups. Let
\[
F = F^{\times} \times \mathcal F (\mathcal C) \quad \text{and} \quad
B=\{\epsilon S\in F \mid \iota(\epsilon)+\sigma(S)\in \mathsf E(C)\} \,,
\]
where $F^{\times} = G_0$,  $\iota\colon F^{\times}\rightarrow C_0$ is a monomorphism, and $\sigma \colon \mathcal F ( \mathcal C) \to \mathcal C$ is the sum function, which is  defined as $\sigma (g_1 \cdot \ldots \cdot g_{\ell}) = g_1 + \ldots + g_{\ell}$. Then
$B\subset F$ is a half-factorial dense seminormal  \C-monoid such that every class of $\mathcal C(B,F)$ contains an element from $C$,
\[
\mathcal C^*(B,F)=\mathcal C(B,F) \cong C, \  \ \mathcal C_{F^{\times}}(B,F)\cong F^{\times},  \quad \text{and} \quad  \mathcal C_B(B,F)=\mathsf E \big( \mathcal C^* (B,F) \big) \,.
\]
\end{example}

\begin{proof}

We start with the following four assertions. Let $\epsilon, \epsilon_1, \epsilon_2 \in F^{\times}$, $S, S_1, S_2 \in \mathcal F (\mathcal C)$, and consider the epimorphism
\[
\lambda\colon F^{\times}\times \mathcal F(C)\rightarrow C \text{ defined by } \lambda(\epsilon S)=\iota(\epsilon)+\sigma(S) \text{ for all }\epsilon S\in F^{\times}\times \mathcal F(C)\,.
\]
\begin{enumerate}
\item[{\bf A1.}\,] $\epsilon_1S_1\sim \epsilon_2 S_2$ if and only if $\lambda(\epsilon_1 S_1)=\lambda(\epsilon_2 S_2)$. In particular, $[\epsilon S] =  [\lambda (\epsilon S)]$.

\smallskip

\item[{\bf A2.}\,] The map $\psi \colon \mathcal C^* (B,F) \to \mathcal C$, defined by $[\epsilon S] \mapsto \lambda (\epsilon S)$, is a semigroup isomorphism.

\smallskip

\item[{\bf A3.}\,]  $\mathcal C_{F^{\times}}(B,F)\cong F^{\times}$ and $\mathcal C^*(B,F)=\mathcal C(B,F)$.

\smallskip

\item[{\bf A4.}\,] $\mathcal C_B(B,F)=\mathsf E \big( \mathcal C^* (B,F) \big)$.
\end{enumerate}

Suppose these four statements hold true. Since $B^{\times} = \{1\} = F^{\times} \cap B$, {\bf A2} implies that $B \subset F$ is a C-monoid. Clearly, $B \subset F$ is dense and by Theorem \ref{1.1} $B$ is seminormal.
Since every class of $\mathcal C^* (B,F)$ contains a prime from $C$,  all assumptions of Theorem \ref{4.2} are satisfied and we obtain that $H$ is half-factorial.
\smallskip

{\it Proof of \,{\bf A1}}.\,  By definition of $B$,  $\lambda(\epsilon_1 S_1)=\lambda(\epsilon_2 S_2)$ implies that $\epsilon_1S_1\sim \epsilon_2 S_2$. Conversely, suppose $\epsilon_1S_1\sim \epsilon_2 S_2$ and $\lambda(\epsilon_1 S_1)\neq\lambda(\epsilon_2 S_2)$.
 If there exists an $i\in [0,n]$ such that $\lambda(\epsilon_1 S_1),\lambda(\epsilon_2 S_2)\in C_i$, then $a=-\lambda(\epsilon_1 S_1)\in \mathcal C \subset F$ and hence $\epsilon_1 S_1 a\in B$, $\epsilon_2 S_2 a\not\in B$, a contradiction. After renumbering if necessary we may assume that $\lambda(\epsilon_1 S_1)\in C_i,\lambda(\epsilon_2 S_2)\in C_j$ for some $i,j\in [0,n]$ with $i<j$. Let $a=-\lambda(\epsilon_1 S_1)\in \mathcal C \subset F$. Since $G_i\supsetneq G_j$ and $C_i+e_j=C_j$, there exists $g_i\in C_i\setminus \{e_i\}$ such that $g_i+e_j=e_j$.
 Then  $\epsilon_1 S_1 a\in B$ which implies that $\epsilon_2 S_2 a\in B$. Thus $\epsilon_2 S_2 a g_i\in B$ which implies that $\epsilon_1 S_1 a g_i\in B$. But $\lambda(\epsilon_1 S_1 a g_i)=g_i\neq e_i$, a contradiction.

To verify the in particular statement, note that $\lambda (\epsilon S) = \iota (\epsilon) + \sigma (S) \in \mathcal C \subset  F$ and for an element $g \in \mathcal C$ we have $\lambda (g) = g$. Thus the claim follows immediately from the main statement.

{\it Proof of \,{\bf A2}}.\, By {\bf A1}, $\psi$ is a well-defined monomorphism and obviously $\psi$ is surjective.

{\it Proof of \,{\bf A3}}.\, Since $\iota\colon F^{\times}\rightarrow C_0$ is a monomorphism, we have, using \eqref{reduced},  that $F^{\times} = F^{\times}/B^{\times}\cong \mathcal C_{F^{\times}}(B,F)\subset C_0$ which implies that $\mathcal C(B,F)=\mathcal C^*(B,F)$.

{\it Proof of \,{\bf A4}}.\, Theorem \ref{1.1} implies that $\mathcal C_B(B,F) \subset \mathsf E \big( \mathcal C^* (B,F) \big)$. Conversely, let $\epsilon S \in F $ such that $[\epsilon S] \in \mathsf E \big( \mathcal C^* (B,F) \big)$. Then {\bf A2} implies that $\iota (\epsilon) + \sigma (S) \in \mathsf E ( \mathcal C)$. Thus, by the definition of $B$, it follows that $\epsilon S \in B$ whence $[\epsilon S] \in \mathcal C_B (B,F)$.
\end{proof}

\providecommand{\bysame}{\leavevmode\hbox to3em{\hrulefill}\thinspace}
\providecommand{\MR}{\relax\ifhmode\unskip\space\fi MR }
% \MRhref is called by the amsart/book/proc definition of \MR.
\providecommand{\MRhref}[2]{%
  \href{http://www.ams.org/mathscinet-getitem?mr=#1}{#2}
}
\providecommand{\href}[2]{#2}


\begin{thebibliography}{10}

\bibitem{Ba-Ch14a}
P.~Baginski and S.T. Chapman, \emph{Arithmetic {C}ongruence {M}onoids: a
  {S}urvey}, Combinatorial and {A}dditive {N}umber {T}heory: {CANT} 2011 and
  2012, Proceedings in Mathematics and Statistics, vol. 101, Springer, 2014,
  pp.~15 -- 38.

\bibitem{Ba94}
V.~Barucci, \emph{Seminormal {M}ori domains}, Commutative {R}ing {T}heory,
  Lect. Notes Pure Appl. Math., vol. 153, Marcel Dekker, 1994, pp.~1 -- 12.

\bibitem{Ch-Co00}
S.T. Chapman and J.~Coykendall, \emph{Half-factorial domains, a survey},
  Non-{N}oetherian {C}ommutative {R}ing {T}heory, Mathematics and {I}ts
  {A}pplications, vol. 520, Kluwer {A}cademic {P}ublishers, 2000, pp.~97 --
  115.

\bibitem{Co05a}
J.~Coykendall, \emph{Extensions of half-factorial domains{\rm \,:} a survey},
  Arithmetical {P}roperties of {C}ommutative {R}ings and {M}onoids, Lect. Notes
  Pure Appl. Math., vol. 241, Chapman \& Hall/CRC, 2005, pp.~46 -- 70.

\bibitem{Co-Ma-Ok17a}
J.~Coykendall, P.~Malcolmson, and F.~Okoh, \emph{Inert primes and factorization
  in extensions of quadratic orders}, Houston J. Math. \textbf{43} (2017), 61
  -- 77.

\bibitem{Cz-Do-Ge16a}
K.~Cziszter, M.~Domokos, and A.~Geroldinger, \emph{The interplay of invariant
  theory with multiplicative ideal theory and with arithmetic combinatorics},
  Multiplicative {I}deal {T}heory and {F}actorization {T}heory, Springer, 2016,
  pp.~43 -- 95.

\bibitem{Do-Fo87}
D.E. Dobbs and M.~Fontana, \emph{Seminormal rings generated by algebraic
  integers}, Mathematika \textbf{34} (1987), 141 -- 154.

\bibitem{Fo-Ha06a}
A.~Foroutan and W.~Hassler, \emph{Factorization of powers in $\rm{C}$-monoids},
  J. Algebra \textbf{304} (2006), 755 -- 781.

\bibitem{Ge-HK04a}
A.~Geroldinger and F.~Halter-Koch, \emph{Congruence monoids}, Acta Arith.
  \textbf{112} (2004), 263 -- 296.

\bibitem{Ge-HK06a}
\bysame, \emph{Non-{U}nique {F}actorizations. {A}lgebraic, {C}ombinatorial and
  {A}nalytic {T}heory}, Pure and Applied Mathematics, vol. 278, Chapman \&
  Hall/CRC, 2006.

\bibitem{Ge-Ka-Re15a}
A.~Geroldinger, F.~Kainrath, and A.~Reinhart, \emph{Arithmetic of seminormal
  weakly {K}rull monoids and domains}, J. Algebra \textbf{444} (2015), 201 --
  245.

\bibitem{Ge-Ra-Re15c}
A.~Geroldinger, S.~Ramacher, and A.~Reinhart, \emph{On $v$-{M}arot {M}ori rings
  and $\rm{C}$-rings}, J. Korean Math. Soc. \textbf{52} (2015), 1 -- 21.

\bibitem{Ge-Zh16c}
A.~Geroldinger and Q.~Zhong, \emph{The set of distances in seminormal weakly
  {K}rull monoids}, J. Pure Appl. Algebra \textbf{220} (2016), 3713 -- 3732.

\bibitem{Gi06a}
R.~Gilmer, \emph{Some questions for further research}, Multiplicative {I}deal
  {T}heory in {C}ommutative {A}lgebra, Springer, New York, 2006, pp.~405--415.

\bibitem{Gr01}
P.A. Grillet, \emph{Commutative {S}emigroups}, Kluwer Academic Publishers,
  2001.

\bibitem{HK09a}
F.~Halter-Koch, \emph{Clifford semigroups of ideals in monoids and domains},
  Forum Math. \textbf{21} (2009), 1001 -- 1020.

\bibitem{HK-Ha-Ka04}
F.~Halter-Koch, W.~Hassler, and F.~Kainrath, \emph{Remarks on the
  multiplicative structure of certain one-dimensional integral domains}, Rings,
  {M}odules, {A}lgebras, and {A}belian {G}roups, Lect. Notes Pure Appl. Math.,
  vol. 236, Marcel Dekker, 2004, pp.~321 -- 331.

\bibitem{Ka16b}
F.~Kainrath, \emph{Arithmetic of {M}ori domains and monoids{\rm \,:} {T}he
  {G}lobal {C}ase}, Multiplicative {I}deal {T}heory and {F}actorization
  {T}heory, Springer Proc. Math. Stat., vol. 170, Springer, 2016, pp.~183 --
  218.

\bibitem{Ma-Ok16a}
P.~Malcolmson and F.~Okoh, \emph{Half-factorial subrings of factorial domains},
  J. Pure Appl. Algebra \textbf{220} (2016), 877 -- 891.

\bibitem{Oh18a}
J.~Oh, \emph{On the algebraic and arithmetic structure of the monoid of
  product-one sequences}, J. Commut. Algebra, to appear,
  https://projecteuclid.org/euclid.jca/1523433705.

\bibitem{Oh19a}
\bysame, \emph{On the algebraic and arithmetic structure of the monoid of
  product-one sequences {II}}, Periodica Math. Hungarica, to appear
  {https://arxiv.org/abs/1802.02851}.

\bibitem{Ph12b}
A.~Philipp, \emph{A precise result on the arithmetic of non-principal orders in
  algebraic number fields}, J. Algebra Appl. \textbf{11}, 1250087, 42pp.

\bibitem{Re13a}
A.~Reinhart, \emph{On integral domains that are $\rm{C}$-monoids}, Houston J.
  Math. \textbf{39} (2013), 1095 -- 1116.

\bibitem{Sc05c}
W.A. Schmid, \emph{Half-factorial sets in finite abelian groups{\rm \,:} a
  survey}, Grazer Math. Ber. \textbf{348} (2005), 41 -- 64.

\bibitem{Sc16a}
\bysame, \emph{Some recent results and open problems on sets of lengths of
  {K}rull monoids with finite class group}, Multiplicative {I}deal {T}heory and
  {F}actorization {T}heory, Springer Proc. Math. Stat., vol. 170, Springer,
  2016, pp.~323 -- 352.

\end{thebibliography}
\end{document}